\documentclass[11pt]{amsart}
\usepackage{amsmath, amssymb, mathrsfs, amsthm, MnSymbol, xifthen}
\usepackage[arrow, matrix]{xy}
\newtheorem{thm}{Theorem}[section]
\newtheorem{lemma}[thm]{Lemma}
\newtheorem{prop}[thm]{Proposition}
\newtheorem{crl}[thm]{Corollary}

\theoremstyle{definition}
\newtheorem{dfn}[thm]{Definition}
\newtheorem{exm}[thm]{Example}

\newtheorem{rem}[thm]{Remark}

\newcommand{\reals}{\mathbb{R}}
\newcommand{\naturals}{\mathbb{N}}
\newcommand{\integers}{\mathbb{Z}}
\newcommand{\rx}{\sgr{\mathbb{R}}{\ux}}
\newcommand{\sos}{\sum\mathbb{R}[\underline{X}]^2}

\newcommand{\spr}{\textrm{Sper}}

\newcommand{\ux}{\underline{X}}

\newcommand{\Pos}{\mbox{Psd}}
\newcommand{\la}{\langle}
\newcommand{\ra}{\rangle}

\newcommand{\ringsos}[1]{\sum #1^2}
\newcommand{\K}[1]{\mathcal{K}_{#1}}
\newcommand{\pos}[1]{\mbox{Psd}(#1)}

\newcommand{\V}[1]{\mathcal{X}_{#1}}
\newcommand{\Z}[1]{\mathcal{Z}(#1)}
\newcommand{\T}[1]{\mathcal{T}_{#1}}

\newcommand{\sgr}[2]{#1[#2]}
\newcommand{\ringsop}[2]{\sum #1^{#2}}

\newcommand{\Hom}[2]{\mbox{Hom}(#1,#2)}
\newcommand{\norm}[2]{\|\ifthenelse{\isempty{#2}}{\cdot}{#2}\|_{#1}}
\newcommand{\cl}[2]{\overline{#2}^{\ifthenelse{\isempty{#1}}{}{#1}}}

\begin{document}

\title[Closure of $\ringsop{R}{2d}$ in Topological Algebras]{Closure of the cone of sums of $2d$-powers in real topological algebras}

\author[M. Ghasemi, S. Kuhlmann]{Mehdi Ghasemi$^1$, Salma Kuhlmann$^2$}

\date{March 23, 2012}

\address{$^1$Department of Mathematics and Statistics,\newline
University of Saskatchewan,\newline
Saskatoon, SK. S7N 5E6, Canada}
\address{$^2$Fachbereich Mathematik und Statistik,\newline
Universit\"{a}t Konstanz\newline 78457 Konstanz, Germany}
\email{mehdi.ghasemi@usask.ca, salma.kuhlmann@uni-konstanz.de}


\begin{abstract}
Let $R$ be a unitary commutative $\reals$-algebra and
$K\subseteq\V{R}:=\Hom{R}{\reals}$, closed with respect to the
product topology. We consider $R$ endowed with the topology $\T{K}$,
induced by the family of seminorms $\rho_{\alpha}(a):=|\alpha(a)|$,
for $\alpha\in K$ and $a\in R$. In case $K$ is compact, we also
consider the topology induced by $\norm{K}{a}:=\sup_{\alpha\in K}|\alpha(a)|$ for $a\in R$. If $K$ is Zariski dense, then those
topologies are Hausdorff. In this paper we prove that the closure of
the cone of sums of $2d$-powers, $\ringsop{R}{2d}$, with respect to
those two topologies is equal to $\pos{K}:=\{a\in R:\alpha(a)\ge0,\textrm{ for all }\alpha\in K\}$. In particular, any
continuous linear functional $L$ on the polynomial ring
$R=\rx:=\reals[X_1,\dots,X_n]$ with $L(h^{2d})\ge0$ for each
$h\in\rx$ is integration with respect to a positive Borel measure
supported on $K$. Finally we give necessary and sufficient conditions to
ensure the continuity of a linear functional with respect to those two topologies.
\end{abstract}

\maketitle
\section{Introduction}
The (real) multidimensional $K$-moment problem for a given closed set
$K\subseteq\reals^n$, is the question of when a real valued linear functional
$L$, defined on the real algebra of polynomials $\rx$, is representable as integration with
respect to a positive Borel measure on $K$. A subset $C$ of $\rx$ is
called a cone, if $C+C\subseteq C$ and $\reals^+ C\subseteq C$, where $\reals^+$ denotes the non-negative real numbers. Let
us denote the cone of non-negative polynomials on $K$ by $\pos{K}$. If $L$ is representable by a measure then clearly,
for any polynomial $f\in\pos{K}$, $L(f)\ge0$ (i.e. $L(\pos{K})\subseteq\reals^+$). Haviland \cite{Hav1, Hav2}, proved that
this necessary condition is also sufficient. However, $\pos{K}$ is
seldom finitely generated \cite[Proposition 6.1]{Sch1}. So in general, there is no
practical decision procedure for the membership problem for
$\pos{K}$, and a fortiori for $L(\pos{K})\subseteq\reals^+$. 

We are mainly interested in the solutions of
\begin{equation}\label{KMP}
    \pos{K}\subseteq\cl{\tau}{C},
\end{equation}
where $C$ is a cone, $K$ is a closed subset of $\reals^n$ and
$\cl{\tau}{C}$ denotes the closure of $C$ with respect to a locally
convex topology $\tau$ on $\rx$. It is proved in \cite[Proposition
3.1]{GKS} that \eqref{KMP} holds if and only if for every
$\tau$-continuous linear functional $L$, nonnegative on $C$, there
exists a positive Borel measure $\mu$ supported on $K$, such that
\[
    L(f)=\int_Kf~d\mu,\quad\forall f\in\rx.
\]
Clearly, if the functional $L$ is representable by a measure, then
$L$ has to be {\it positive semidefinite}, i.e., $L(p^2)\ge0$ for
all $p\in\rx$. So in \eqref{KMP}, the simplest cone to consider is
$C=\sos$. If $\tau$ is the finest locally convex topology $\varphi$
on $\rx$, then any given linear functional is $\varphi$-continuous
\cite{BCRBK}. In this setting, Schm\"{u}dgen in \cite[Theorem
3.1]{Schm2} and Berg, Christensen and Jensen in \cite[Theorem
3]{BCJ} prove that $\cl{\varphi}{\sos}=\sos$. Later, in \cite[Theorem 9.1]{BCR}
Berg, Christensen and Ressel prove that taking $C=\sos$ and $\tau$
to be the $\ell_1$-norm topology, then $K=[-1,1]^n$ solves
\eqref{KMP}, i.e., $\cl{\norm{1}{}}{\sos}=\pos{[-1,1]^n}$. This was
further generalized in \cite{BCRBK} and \cite{BM-EB} to include
commutative semigroup-rings and topologies induced by absolute
values. These results has been revisited in \cite{JLN} with a
different approach, and were recently generalized in \cite{GKS} to
weighted $\ell_p$-norms, $p\ge1$.
In \cite{GMW} it is shown that the general result in \cite{BM-EB}
carries to the even smaller cone of sums of $2d$-powers,
$\ringsop{\rx}{2d}\subseteq\sos$, where $d\ge1$ is an integer.  In
particular, $\cl{\norm{1}{}}{\ringsop{\rx}{2d}}=\pos{[-1,1]^n}$. 

We now discuss \eqref{KMP} for other special cones. A set
$T\subseteq\rx$ is called a \textit{preordering}, if $T+T\subseteq
T$, $T\cdot T\subseteq T$ and $\sos\subseteq T$. For
$S\subseteq\rx$, we denote the smallest preordering, containing $S$
by $T_S$. A preordering $T$ is said to be finitely generated, if
$T=T_S$ for some finite set $S\subseteq\rx$. The smallest
preordering of $\rx$ is $\sos$, as considered above. A subset
$M\subseteq\rx$ is a $\ringsop{\rx}{2d}$-\textit{module} if $1\in
M$, $M+M\subseteq M$ and $\ringsop{\rx}{2d}\cdot M\subseteq M$. If
$d=1$, $M$ is said to be a {\it quadratic} module.  $M$ is said to
be finitely generated, if $M=M_S$ for some finite set
$S\subseteq\rx$, and \textit{Archimedean} if for every $f\in\rx$
there exists $n\in\naturals$ such that $n\pm f\in M$. The
non-negativity set of a subset $S\subset \rx$ will be denoted by
$K_S$, and is defined by $K_S:=\{x\in\reals^n:\forall f\in
S~f(x)\ge0\}$. If $S$ is finite, $K_S$ is called a \textit{basic
closed semialgebraic set}. 

In \cite{Schm} Schm\"{u}dgen proves
that for a finite $S$, if $K_S$ is \textit{compact}, then $K_S$
solves \eqref{KMP} for $C=T_S$ and $\tau=\varphi$, i.e.,
$\cl{\varphi}{T_S}=\pos{K_S}$.
Jacobi proved \cite{J} that if $M$ is an Archimedean
$\ringsop{\rx}{2d}$-module then $\cl{\varphi}{M}=\pos{K_M}$ (see \cite[Theorem 1.3 and 1.4]{MPut} for $d=1$). In
\cite{JL:K-Moment}, Lasserre proves that for a specific fixed norm
$\norm{w}{}$, and any finite $S$,
$\cl{\norm{w}{}}{M_S}=\cl{\norm{w}{}}{T_S}=\pos{K_S}$. In
particular, if $S=\emptyset$,
$\cl{\norm{w}{}}{\sos}=\pos{\reals^n}$. It is worth noting that the
latter equality could be derived by suitable modifications of
Schm\"{u}dgen \cite[Lemma 6.1 and Lemma 6.4]{Schm2}.

\textit{Throughout the paper the algebras under consideration are unitary and commutative}. In this paper, we study \eqref{KMP} in a more general context.
In Section \ref{Background}, we recall some standard notations and
elementary material which will be needed in the following sections.
We consider a $\integers[\frac{1}{2}]$-algebra $R$ and a
$K\subseteq\V{R}:=\Hom{R}{\reals}$, closed with respect to the
product topology. 

In Section \ref{topologies}, we associate to $K$ a
topology $\T{K}$ on $R$, making all homomorphisms in $K$ continuous.
When $K$ is compact we define a seminorm $\norm{K}{}$ on $R$, which
induces a strictly finer topology than $\T{K}$. If $K$ is Zariski
dense, then those topologies are Hausdorff. 

In section \ref{MP}, we study \eqref{KMP} in
terms of the two topologies $\T{K}$, $\norm{K}{}$ for the cone
$C=\ringsop{R}{2d}$ of sums of $2d$-powers. 
The two main results are Theorems \ref{SosRng} and \ref{SosRngNC}:
we prove that for $K$ as above, the closure of $\ringsop{R}{2d}$
with respect to $\T{K}$ is $\pos{K}$. 
Here $\pos{K}:=\{a\in R:\alpha(a)\ge0,\textrm{ for all }\alpha\in K\}$.
When $K$ is compact, we use Stone-Weierstrass to prove that the
closure of $\ringsop{R}{2d}$ with respect to the
$\norm{K}{}$-topology is again $\pos{K}$. 
In case $R=\rx$, $K=\reals^n$ and
$d=1$, the first result is a special case of Schm\"{u}dgen's result for locally
multiplicatively convex topologies \cite[Proposition 6.2]{Schm2},
and the second result is straightforward, as noted in \cite[Remark 3.2]{BCJ}. 

Finally, we apply our results to obtain representation of continuous functionals by
measures (Corollaries \ref{M4Cmpt} and \ref{M4NCmpt}).

In Section \ref{RAlg}, we study the case when $R$ is an $\reals$-algebra. We define $\ringsop{R}{2d}$-modules and archimedean modules exactly as we did for $R=\rx$.
We prove that the closure of $\ringsop{R}{2d}$ with respect to
any sub-multiplicative norm is $\pos{\K{\norm{}{}}}$, where $\K{\norm{}{}}$ is the Gelfand spectrum of $(R,\norm{}{})$ (see Theorem \ref{ClsrNrmdAlg}).
Our proof is algebraic and uses a result of T. Jacobi \cite[Theorem 4]{J}. Again, we get representation of continuous functionals by measures (Corollary \ref{M4NAlge}).

Next, we study the case where the cone is a $\ringsop{R}{2d}$-module $M\subseteq R$.
We define $\K{M}=\{\alpha\in\V{R}: \alpha(a)\ge0\quad \forall a\in M\}$.
We show that $\cl{\T{\V{R}}}{M}$, the closure of $M$ with respect to $\T{\V{R}}$, is $\pos{\K{M}}$ (see Theorem \ref{ClMdl}).

In Section \ref{App2RX}, we apply
all these results to the ring of polynomials $\rx$. Moreover, we
study the case when $K$ is not necessarily Zariski-dense. We show
that if $K$ is contained in a variety, a locally convex and
Hausdorff topology $\tau_K$ can still be defined, as the limit of an
inverse family of topologies on $\rx$. We show that
$\cl{\tau_K}{\ringsop{\rx}{2d}}=\tilde{\Pos}(K)$, where
$\tilde{\Pos}(K)$ is the set of polynomials which are nonnegative on
some open set containing $K$. Finally, we compare 
the topologies $\norm{K}{}$ and $\T{K}$ on $\rx$ to sub-multiplicative norm topologies, 
and to the Lasserre's topology $\norm{w}{}$, considered on \cite{JL:K-Moment}.
\section{Preliminaries on topological vector spaces and
rings.}\label{Background}
In the following, all vector spaces are over the field of real
numbers (unless otherwise specified). A \textit{topological vector
space} is a vector space $X$ equipped with a topology such that the
vector space operations (i.e. scalar multiplication and vector
summation) are continuous. A subset $A\subseteq X$ is said to be
\textit{convex} if for every $x,y\in A$ and $\lambda\in[0,1]$,
$\lambda x+(1-\lambda)y\in A$. A \textit{locally convex}
(\textit{lc} for short) topology is a topology which admits a
neighborhood basis of convex open sets at each point. 

Suppose that in addition $X$ is an $\reals$-algebra. A subset $U\subseteq X$
is called a \textit{multiplicative set}, an $m$\textit{-set} for
short, if $U\cdot U\subseteq U$. A locally convex topology on $X$ is
said to be \textit{locally multiplicatively convex} (or \textit{lmc}
for short) if there exists a fundamental system of neighborhoods for
$0$ consisting of $m$-sets. It is immediate from the definition that
the multiplication is continuous in a lmc-topology. 

\begin{dfn}
A function $\rho:X\rightarrow\reals^{\ge0}$ is called a seminorm, if\\
\indent (i) $\forall x,y\in X\quad\rho(x+y)\leq\rho(x)+\rho(y)$,\\
\indent (ii) $\forall x\in X~\forall r\in\reals\quad\rho(rx)=|r|\rho(x)$.\\
$\rho$ is called a multiplicative seminorm, if in addition $\rho$ satisfies the following\\
\indent (iii) $\forall x,y\in X\quad\rho(x\cdot y)\leq\rho(x)\rho(y).$
\end{dfn}
\begin{dfn}
Let $\mathcal{F}$ be a nonempty family of seminorms on $X$. The
topology generated by $\mathcal{F}$ on $X$ is the coarsest topology
on $X$ making all seminorms in $\mathcal{F}$ continuous. It is a
locally convex topology on $X$. The family of sets of the form
\[
    U^{\epsilon}_{\rho_1,\dots,\rho_k}(x):=\{y\in X~:~\rho_i(x-y)\leq\epsilon,~i=1,\dots,k\}
\]
where $\epsilon>0$ and $\rho_1\dots,\rho_k\in\mathcal{F}$, forms a
basis for this topology.

\end{dfn}~\\
 We have the following characterization of lc and lmc
spaces.
\begin{thm}
Let $X$ be an algebra and $\tau$ a topology on $X$. Then
\begin{enumerate}
    \item{$\tau$ is lc if and only if it is generated by a family of seminorms on $X$.}
    \item{$\tau$ is lmc if and only if it is generated by a family of multiplicative seminorms on $X$.}
\end{enumerate}
\end{thm}
\begin{proof}
See \cite[Theorem 6.5.1]{HJ} for (1) and \cite[4.3-2]{B-N-S} for (2).
\end{proof}

Let $R$ be a commutative ring with $1$ and $\frac{1}{2}\in R$.
We always assume that $\V{R}=\Hom{R}{\reals}$, the set of unitary
homomorphisms, is nontrivial. Clearly, $\V{R}\subset\reals^R$,
therefore, it carries a topology as subspace of $\reals^R$ with
the product topology which is Hausdorff. For any $a\in R$ let $U(a):=\{\alpha\in\V{R}:
\alpha(a)<0\}$. The family $\{U(a):a\in R\}$ forms a subbasis for
the subspace topology on $\V{R}$ which is the coarsest topology
making all projection functions $\hat{a}:\V{R}\rightarrow\reals$ continuous where for
$a\in R$, $\hat{a}$ is defined by $\hat{a}(\alpha)=\alpha(a)$. $\V{R}$
also can be embedded in $\spr(R)$ equipped with spectral topology
\cite[Theorem 5.2.5 and Lemma 5.2.6]{MPS}. Since all projections are
continuous, the topology of $\V{R}$ coincides with the subspace
topology inherited from $\reals^R$ equipped with product topology.
For $K$ be a subset of $\V{R}$  we denote by $C(K)$ the algebra of continuous real valued
functions on $K$.

\begin{dfn}
To any subset $S$ of $R$ we associate a subset $\Z{S}$ of $\V{R}$,
called the zeros of $S$ or the variety of $S$ by
$\Z{S}:=\{\alpha\in\V{R}:\alpha(S)=\{0\}\}$. Denote by $\la S\ra$
the ideal generated by $S$. Then $\Z{S}=\Z{\la S\ra}$, and the
family $\{\V{R}\setminus\Z{S}:S\subseteq R\}$ forms a sub-basis for
a topology called the \textit{Zariski topology} on $\V{R}$.
\end{dfn}

Now Suppose that the map (defined in Lemma \ref{phidef} with $K = \V{R}$,
$a\mapsto\hat{a}$ is injective, then a subset $K$ of $\V{R}$ is
dense in $\V{R}$ with respect to the Zariski topology
(\textit{Zariski dense}) if and only if $K\not\subset\Z{I}$ for any
proper $I$ ideal of $R$.

\begin{exm}
Suppose that $K\subseteq\reals^n$ has nonempty interior, then
clearly $K$ is Zariski dense in $\reals^n$. Let $K$ be a basic
closed semialgebraic set, i.e., $K:=\{x\in\reals^n:f_i(x)\ge0,
i=1,\dots,m\}$ where $f_i\in\rx$, $i=1,\dots,m$. We show that if $K$
has an empty interior then $K$ is contained in $\Z{g}$, where
$g=f_1\dots f_m$. For each $x\in K$, $f_i(x)=0$ for some $1\leq
i\leq m$. Otherwise, $f_i(x)>0$ for each $i=1,\dots,m$, the
continuity of polynomials implies that $f_i(y)>0$, $i=1,\dots,m$ for
$y$ sufficiently close to $x$ and hence $x$ is an interior point
which is impossible. Therefore $g(x)=0$ and hence $K\subseteq\Z{g}$.
Note that any closed semialgebraic set $K$, is a finite union of
basic closed semialgebraic sets \cite[Theorem 2.7.2]{B-C-R} i.e.,
$K=\bigcup_{i=1}^lK_i$. If $K$ has an empty interior, then each
$K_i$ is so and hence $K\subseteq\Z{g_1\dots g_l}$, where each $g_i$
is the product of generators of $K_i$.
\end{exm}
\begin{rem}
In general, the conclusion of the above example is false. For example, let $R=\integers_3\times\rx$ with component wise addition and multiplication. 
$R$ is a $\integers[\frac{1}{2}]$-module where $\frac{1}{2}(1,r)=(2,\frac{r}{2})$ and $\frac{1}{2}(2,r)=(1,\frac{r}{2})$. Also $\V{R}=\reals^n$ is non-trivial. 
$I=\integers_3\times\{0\}$ is a proper ideal of $R$ and $\Z{I}=\reals^n$. Hence, every semialgebraic set is contained in $\Z{I}$.
\end{rem}

\begin{dfn}
If $R$ is a ring, a function $N:R\rightarrow\reals^+$ is called a
\textit{ring-seminorm}
if the following conditions hold for all $x,y\in R$:\\
    \indent (i) $N(0)=0$,\\
    \indent (ii) $N(x+y)\leq N(x)+N(y)$,\\
    \indent (iii) $N(-x)=N(x)$,\\
    \indent (iv) $N(xy)\leq N(x)N(y)$.\\
$N$ is called a \textit{ring-norm}, if in addition\\
    \indent (v) $N(x)=0$ only if $x=0$.
\end{dfn}

We close this section by stating a general version of Haviland's
Theorem.
\begin{thm}\label{Haviland}
Suppose $R$ is an $\reals$-algebra, $X$ is a  Hausdorff space, and
$~\hat{}:R\rightarrow C(X)$ is an $\reals$-algebra homomorphism such
that for some $p\in R$, $\hat{p}\ge0$ on $X$ and the set $X_i=\hat{p}^{-1}([0,i])$ is compact
for each $i=1,2,\dots$. Then for every linear functional
$L:R\rightarrow\reals$ satisfying
\[
    L(\{a\in R:\hat{a}\ge0\textrm{ on }X\})\subseteq\reals^+,
\]
there exists a Borel measure $\mu$ on $X$ such that $\forall a\in
R\quad L(a)=\int_X\hat{a}~d\mu$.
\end{thm}
\begin{proof}
See \cite[Theorem 3.2.2]{MPS}.
\end{proof}
\section{The topologies $\T{K}$ and $\|.\|_{K}$.}\label{topologies}
Throughout we assume that the map $\hat{~~}:R\rightarrow C(\V{R})$, defined by $\hat{a}(\alpha)=\alpha(a)$ is injective.
\begin{lemma}\label{phidef}
Let $K$ be a subset of $\V{R}$, then
\begin{enumerate}
    \item{The map $\Phi:R\rightarrow C(K)$ defined by $\Phi(a)=\hat{a}|_K$ is a homomorphism.}
    \item{$Im(\Phi)$ contains a copy of $\integers[\frac{1}{2}]$.}
\end{enumerate}
\end{lemma}
\begin{proof}
(1) This is clear. Let $a,b\in R$, then for each $\alpha\in K$ we have
\[
    \begin{array}{lll}
    \Phi(a+b)(\alpha) & = & \widehat{(a+b)}|_K(\alpha)\\
     & = & \alpha(a+b)\\
     & = & \alpha(a)+\alpha(b)\\
     & = & \hat{a}|_K(\alpha)+\hat{b}|_K(\alpha)\\
     & = & \Phi(a)(\alpha)+\Phi(b)(\alpha).
    \end{array}
\]
Similarly $\Phi(a\cdot b)=\Phi(a)\cdot\Phi(b)$. 

\noindent (2) Since
$\V{R}$ consists of unitary homomorphisms,
$\hat{1}(\alpha)=\alpha(1)=1$, so the constant function
$1\in\Phi(R)$. Moreover for each $m\in\integers$ and
$n\in\naturals$, $\frac{m}{2^n}\in R$ and $\hat{\frac{m}{2^n}}$ is
the constant function $\frac{m}{2^n}$ which belongs to $Im(\Phi)$,
so $\integers[\frac{1}{2}]\subseteq Im(\Phi)$.
\end{proof}
\subsection*{The topology $\T{K}$}
Let $K\subseteq\V{R}$. To any $\alpha \in K$ we
associate a seminorm $\rho_{\alpha}$ on $C(K)$ by defining
$\rho_{\alpha}(f):=|f(\alpha)|$ for $f\in C(K)$. Note that
$\rho_{\alpha}(fg)=\rho_{\alpha}(f)\rho_{\alpha}(g)$, so
$\rho_{\alpha}$ is a multiplicative seminorm. The family of
seminorms $\mathcal{F}_K=\{\rho_{\alpha}:\alpha\in K\}$ thus induces
an lmc-topology on $C(K)$.

Similarly, The restriction of
$\rho_{\alpha}$ to $\Phi(R)$ induces a multiplicative ring-seminorm on
$R$ by defining $\rho_{\alpha}(a):=|\hat{a}(\alpha)|=|\alpha(a)|$
for $a \in R$. Thus family of ring-seminorms $\mathcal{F}_K$ induces a
topology $\T{K}$ on $R$.

To ease the notation we shall denote
the neighborhoods by
\[U_{\rho_{\alpha_1},\dots,\rho_{\alpha_m}}^{\epsilon}(a): = U_{\alpha_1,\dots,\alpha_m}^{\epsilon}(a)\>.\]
\begin{rem}\label{FKTK}
 $\T{K}$ is the coarsest topology on $R$ for which all $\alpha\in K$ are
continuous.  $\T{K}$ is also the coarsest topology on $R$, for which
$\Phi$ is continuous. This is clear, because
$\rho_{\alpha}(a)=\rho_{\alpha}(\Phi(a))$ for each $a\in R$. 

We note for future reference that the topology generated by
$\mathcal{F}_K$ on $C(K)$ is Hausdorff \cite[Proposition
1.8]{BCRBK}.
\end{rem}
\begin{thm}\label{ZD-H-Ker}
Let $K\subseteq\V{R}$ and $\Phi:R\rightarrow C(K)$ be the map defined in Lemma \ref{phidef}.
The following are equivalent:
\begin{enumerate}
    \item{$K$ is Zariski dense,}
    \item{$\ker\Phi=\{0\}$,}
    \item{$\T{K}$ is a Hausdorff topology.}
\end{enumerate}
\end{thm}
\begin{proof}
(1)$\Rightarrow$(2) In contrary, suppose that $\ker\Phi\neq\{0\}$ and let $0\neq a\in\ker\Phi$.
Then by definition, $\hat{a}(\alpha)=0$ for all $\alpha\in K$.
This implies $K\subseteq\Z{a}$ which contradicts the assumption that $K$ is Zariski dense.\\
(2)$\Rightarrow$(3)
 Since $\Phi$ is injective,
by Remark \ref{FKTK}, $\Phi$ is a topological embedding. This implies that $\T{K}$ is Hausdorff as well.\\
(3)$\Rightarrow$(1) Suppose that $K$ is not Zariski dense. Than $K\subseteq\Z{I}$ for a nontrivial ideal of $R$.
Take $a,b\in I$, $a\neq b$ and let $U$ be
an open set in $\T{K}$, containing $a$. By definition, there exist $\alpha_1,\dots,\alpha_m\in K$ and $\epsilon>0$
such that
\[
    U_{\alpha_1,\dots,\alpha_m}^{\epsilon}(a)\subseteq U.
\]
For each $i=1,\dots,m$, $\rho_{\alpha_i}(a-b)=|\alpha_i(b-a)|=0$, therefore
$b\in U_{\alpha_1,\dots,\alpha_m}^{\epsilon}(a)$ and hence $b\in U$. This shows that
$\T{K}$ is not Hausdorff, a contradiction.
\end{proof}
\subsection*{The topology $\|.\|_{K}$} Assume now that $K$ is compact.
In this case, $C(K)$ carries a natural norm topology, the norm
defined by $\norm{K}{f}=\sup_{\alpha\in K}|f(\alpha)|$.
The inequalities $\norm{K}{f+g}\leq\norm{K}{f}+\norm{K}{g}$ and
$\norm{K}{fg}\leq\norm{K}{f}\norm{K}{g}$ implies the continuity of
addition and multiplication on $(C(K),\norm{K}{})$.
\begin{lemma}\label{denseness}
If $K\subseteq\V{R}$ is compact, then $\Phi(R)$ is dense in $(C(K),\norm{K}{})$.
\end{lemma}
\begin{proof}
Let $\mathcal{A}=\cl{}{\Phi(R)}$. We make use of Stone-Weierstrass
Theorem to show that $\mathcal{A}=C(K)$. $K$ is compact and
Hausdorff, so once we show that $\mathcal{A}$ is an $\reals$-algebra
which contains all constant functions and separates points of $K$,
we are done (See \cite[Theorem 44.7]{GTW}). Note that $\mathcal{A}$
contains all constant functions because
$\integers[\frac{1}{2}]\subset\Phi(R)$ which is dense in $\reals$.
Since addition and multiplication are continuous, they extend
continuously to $\mathcal{A}$, therefore, $\mathcal{A}$ is also
closed under addition and multiplication. Moreover, $\mathcal{A}$
separates points of $K$, because $\Phi(R)$ does. Hence, by
Stone-Weierstrass Theorem $\mathcal{A}=C(K)$.
\end{proof}

\begin{rem}\label{pullback}~
\begin{itemize}
	\item[1.]{Corresponding to  Remark \ref{FKTK}, defining the ring-norm $\norm{K}{}$ on $R$ by $\norm{K}{a}=\norm{K}{\hat{a}}$ induces
	a topology which is the coarsest topology such that $\Phi$ is continuous.
	But $\norm{K}{}$ is not necessarily a norm, unless when $\Phi$ is injective which by
	Theorem \ref{ZD-H-Ker} is equivalent to $K$ being Zariski dense.
	}
	\item[2.]{For any $\alpha\in K$, the evaluation map at $\alpha$, over
	$C(K,\norm{K}{})$ satisfies the inequality
	$|f(\alpha)|=\rho_{\alpha}(f)\leq\norm{K}{f}$, so it is continuous for each $\alpha\in K$. This
	observation shows that each $\T{K}$-open set is also
	$\norm{K}{}$-open, i.e., $\norm{K}{}$-topology is finer than
	$\T{K}$. We show that if $K$ is infinite, then
	$\norm{K}{}$-topology is strictly finer than $\T{K}$.}
\end{itemize}
\end{rem}

\begin{prop}\label{SrctFin}
If $K$ is an infinite, compact subset of $\V{R}$, then
$\norm{K}{}$-topology is strictly finer than $\T{K}$.
\end{prop}
\begin{proof}
Let $\alpha_1,\dots,\alpha_m\in K$ and $0<\epsilon<1$. We claim that
there exists $a\in R$ such that $a\in
U_{\alpha_1,\dots,\alpha_m}^{\epsilon}(0)$ and
$\norm{K}{a}>\epsilon$. Note that $\V{R}$ is Hausdorff and so is
$K$. Compactness of $K$ implies that $K$ is a normal space. Take
$A=\{\alpha_1,\dots,\alpha_m\}$ and $B=\{\beta\}$, where $\beta\in
K\setminus A$. By Urysohn's lemma, there exists a continuous
function $f:K\rightarrow[0,1]$ such that $f(A)=0$ and $f(\beta)=1$.
For $\delta<\min\{\epsilon,1-\epsilon\}$, there exists $a\in R$ such
that $\norm{K}{f-a}<\delta$ by Lemma \ref{denseness}. Clearly $a\in
U_{\alpha_1,\dots,\alpha_m}^{\epsilon}(0)$ and
$|\hat{a}(\beta)|>\epsilon$ which implies that
$\norm{K}{a}>\epsilon$ which completes the proof of the claim. 

Let $N_{\epsilon}(0)=\{a\in R : \norm{K}{a}<\epsilon\}$ be an open
ball around $0$ in $\norm{K}{}$ for $0<\epsilon<1$. We show that
$N_{\epsilon}(0)$ does not contain any open neighborhood of $0$ in
$\T{K}$. In contrary, suppose that $0\in
U_{\alpha_1,\dots,\alpha_m}^{\delta}(0)\subseteq N_{\epsilon}(0)$.
Obviously $\delta\leq\epsilon$ and so there exists $a\in
U_{\alpha_1,\dots,\alpha_m}^{\delta}(0)$ such that $a\not\in
N_{\epsilon}(0)$ which is a contradiction. So, $N_{\epsilon}(0)$ is
not open in $\T{K}$ and hence, $\norm{K}{}$-topology is strictly
finer that $\T{K}$.
\end{proof}
\section{Closures of $\sum R^{2d}$ in $\mathcal{T}_{K}$ and $\|\cdot\|_{K}$}\label{MP}
In this section, we compute the closure $\ringsop{R}{2d}$ in the
two topologies defined in the previous section. In
particular, for compact $K\subseteq\V{R}$, we show that
\[
	\cl{\norm{K}{}}{\ringsop{R}{2d}}=\cl{\T{K}}{\ringsop{R}{2d}},
\]
although for infinite $K$, the $\norm{K}{}$-topology is strictly finer than $\T{K}$ on
$R$ by Proposition \ref{SrctFin}. Let
\[
    C^+(K):=\{f:K\rightarrow\reals~:~f\textrm{ is continuous and }\forall\alpha\in K\quad f(\alpha)\ge0\},
\]
denote the set of nonnegative real valued continuous functions over
$K$ and \[\pos{K}:=\{a\in R:\hat{a}\in C^+(K)\}\>.\]
\begin{prop}\label{Closeness}
$\pos{K}$ is closed in $\T{K}$. If $K$ is compact, then $\pos{K}$ is also closed in $\norm{K}{}$-topology.
\end{prop}
\begin{proof}
For each $\alpha\in K$, let $e_{\alpha}(f)=f(\alpha)$ be the evaluation map. Then $e_{\alpha}^{-1}([0,\infty))$ is closed, by continuity of $e_{\alpha}$ in $\T{K}$. 
Therefore $\pos{K}=\bigcap_{\alpha\in K}e_{\alpha}^{-1}([0,\infty))$ is closed.
If $K$ is compact, then, again $e_{\alpha}$ is $\norm{K}{}$-continuous. Therefore $\pos{K}$ is also closed with respect to $\norm{K}{}$.
\end{proof}
\begin{thm}\label{SosRng}
For any compact set $K\subseteq\V{R}$ and integer $d\ge1$, $\cl{\norm{K}{}}{\ringsop{R}{2d}}=\pos{K}$.
\end{thm}
\begin{proof}
Since $\ringsop{R}{2d}\subseteq\pos{K}$ and $\pos{K}$ is closed, clearly $\cl{\norm{K}{}}{\ringsop{R}{2d}}\subseteq\pos{K}$.
To show the reverse inclusion, let $a\in\pos{K}$ and $\epsilon>0$ be given. Since $\hat{a}\ge0$ on $K$, $\sqrt[2d]{\hat{a}}\in C(K)$.
Continuity of multiplication implies the continuity of the map $f\mapsto f^{2d}$. Therefore, there exists $\delta>0$ such that
$\norm{K}{\sqrt[2d]{\hat{a}}-f}<\delta$ implies $\norm{K}{\hat{a}-f^{2d}}<\epsilon$.
Using Lemma \ref{denseness}, there is $b\in R$ such that $\norm{K}{\sqrt[2d]{\hat{a}}-\hat{b}}<\delta$ and so $\norm{K}{\hat{a}-\hat{b}^{2d}}<\epsilon$.
By definition, $\hat{a}-\hat{b}^{2d}=\Phi(a-b^{2d})$ and hence $\norm{K}{a-b^{2d}}<\epsilon$. Therefore, any neighborhood of $a$ has nonempty intersection
with $\ringsop{R}{2d}$ which proves the reverse inclusion $\pos{K}\subseteq\cl{\norm{K}{}}{\ringsop{R}{2d}}$.
\end{proof}
\begin{crl}\label{M4Cmpt}
Let $K$ be a compact subset of $\V{R}$, $d\ge1$ an integer. Assume that $L:R\rightarrow\reals$ is $\norm{K}{}$-continuous, $\integers[\frac{1}{2}]$-linear map,
such that $L(a^{2d})\ge0$ for all $a\in R$, then there exists a Borel Measure $\mu$ on $K$ such that $\forall a\in R ~ L(a)=\int_K \hat{a}~d\mu$.
\end{crl}
\begin{proof}
Let $\hat{R}:=\{\hat{a}:a\in R\}$ and define $\bar{L}:\hat{R}\rightarrow\reals$ by $\bar{L}(\hat{a})=L(a)$.

We prove if $\hat{a}\ge0$, then $L(a)\ge0$. To see this, let $\epsilon>0$ be given and find $\delta>0$ such that $\norm{K}{a-b}<\delta$ implies
$|L(a)-L(b)|<\epsilon$. Take $c_{\epsilon}\in R$ such that $\norm{K}{a-c_{\epsilon}^{2d}}<\delta$. Then
\[
    L(c_{\epsilon}^{2d})-\epsilon<L(a)<L(c_{\epsilon})+\epsilon,
\]
let $\epsilon\rightarrow0$, yields $L(a)\ge0$.

Note that $\bar{L}$ is well-defined, since $\hat{a}=0$, implies $\hat{a}\ge0$ and $-\hat{a}\ge0$, so $\bar{L}(\hat{a})\ge0$ and $\bar{L}(-\hat{a})\ge0$,
simultaneously and hence $\bar{L}(\hat{a})=0$.
$\norm{K}{}$-continuity of $L$ on $R$, implies $\norm{K}{}$-continuity of $\bar{L}$ on $\hat{R}$. Let $A$ be the $\reals$-subalgebra of $C(K)$, generated by
$\hat{R}$. Elements of $A$ are of the form $r_1\hat{a_1}+\dots+r_k\hat{a_k}$, where $r_i\in\reals$ and $a_i\in R$, for $i=1\dots,k$ and $k\ge1$.
$\bar{L}$ is continuously extensible to $A$ by
$\bar{L}(r\hat{a}):=r\bar{L}(\hat{a})$. By Lemma \ref{denseness}, $\hat{R}$ and hence $A$ is dense in $(C(X),\norm{K}{})$. Hahn-Banach Theorem gives a continuous
extension of $\bar{L}$ to $C(X)$. Denoting the extension again by $\bar{L}$, an easy verification shows that $\bar{L}(C^+(K))\subseteq\reals^+$.
Applying Riesz Representation Theorem, the result follows.
\end{proof}

\begin{rem}
For the special case $K=\reals^n$ and $R=\rx$, it follows from
\cite[Proposition 6.2]{Schm2} that the closure of $\sos$ with
respect to the finest lmc topology $\eta_0$ on $\rx$ is equal to
$\pos{\reals^n}$. Since $\T{\reals^n}$ is lmc and $\pos{\reals^n}$
is closed in $\T{\reals^n}$ we get
\[
    \pos{\reals^n}=\cl{\eta_0}{\sos}\subseteq\cl{\T{\reals^n}}{\sos}\subseteq\cl{\T{\reals^n}}{\pos{\reals^n}}=\pos{\reals^n}.
\]
In the next theorem, we show that a similar result holds for
arbitrary $K$ and the smaller set of sums of $2d$-powers
$\ringsop{R}{2d}\subset\ringsos{R}$.
\end{rem}
\begin{thm}\label{SosRngNC}
Let $K\subseteq\V{R}$ be a closed set and $d\ge1$, then $\cl{\T{K}}{\ringsop{R}{2d}}=\pos{K}$.
\end{thm}

\begin{proof}
Since $\ringsop{R}{2d}\subseteq\pos{K}$ and by Proposition \ref{Closeness}, $\pos{K}$ is closed, we have $\cl{\T{K}}{\ringsop{R}{2d}}\subseteq\pos{K}$.

To get the reverse inclusion, let $a\in\pos{K}$ be given. We show that any neighborhood of $a$ in $\T{K}$ has a nonempty intersection with $\ringsop{R}{2d}$.

\textit{Claim.} If $\hat{a}>0$ on $K$ then $a\in\cl{\T{K}}{\ringsop{R}{2d}}$.\\
To prove this, let $U$ be an open set, containing $a$. There exist $\alpha_1,\dots,\alpha_n\in K$ and $\epsilon>0$ such that
$a\in U_{\alpha_1,\dots,\alpha_n}^{\epsilon}(a)\subseteq U$. Chose $m\in\naturals$ such that $\max_{1\leq i\leq n}\alpha_i(a)<2^{2dm}$.
Now for $b=\frac{a}{2^{2dm}}$ we have $0<\alpha_i(b)<1$. By continuity of $f(t)=t^{2d}$, for each $i=1,\dots,n$ there exists $\delta_i>0$ such that for any
$t$, if $|t-\alpha_i(b)^{1/2d}|<\delta_i$, then $|t^{2d}-\alpha_i(b)|<\frac{\epsilon}{2^{2dm}}$. Take $\delta=\min_{1\leq i\leq n}\delta_i$.
Let $p(t)=\sum_{j=0}^N\lambda_jt^j$ be the real polynomial satisfying $p(\alpha_i(b))=\sqrt[2d]{\alpha_i(b)}$ for $i=1,\dots,n$.
Since $\integers[\frac{1}{2}]$ is dense in $\reals$ one can choose
$\tilde{\lambda}_j\in\integers[\frac{1}{2}]$, such that $|\sum_{j=1}^{N}\lambda_j\alpha_i(b)^j-\sum_{j=1}^{N}\tilde{\lambda}_j\alpha_i(b)^j|<\delta$,
for $i=1,\dots,n$. Let $c=\sum_{j=1}^{N}\tilde{\lambda}_jb^j\in R$. Then $|\alpha_i(b)-\alpha_i(c)^{2d}|<\frac{\epsilon}{2^{2dm}}$.
Multiplying by $2^{2dm}$, $|\alpha_i(a)-\alpha_i(2^mc)^{2d}|<\epsilon$ i.e. $\rho_{\alpha_i}(a-(2^mc)^{2d})<\epsilon$ for $i=1,\dots,n$. Therefore
$U_{\alpha_1,\dots,\alpha_n}^{\epsilon}(a)\cap\ringsop{R}{2d}\neq\emptyset$ and hence $a\in\cl{\T{K}}{\ringsop{R}{2d}}$ which completes the proof of the claim.

For an arbitrary $a\in\pos{K}$, and each $k\in\naturals$, $\widehat{(a+\frac{1}{2^k})}>0$ on $K$, so
$\forall k\in\naturals,~ a+\frac{1}{2^k}\in\cl{\T{K}}{\ringsop{R}{2d}}$. Letting $k\rightarrow\infty$, $\rho_{\alpha}(a+\frac{1}{2^k})\rightarrow\rho_{\alpha}(a)$,
we get $a\in\cl{\T{K}}{\ringsop{R}{2d}}$ and hence $\pos{K}\subseteq\cl{\T{K}}{\ringsop{R}{2d}}$ as desired.
\end{proof}
\begin{crl}\label{M4NCmpt}
Let $K$ be a closed subset of $\V{R}$ and $d\ge1$ an integer. Assume that there exists $p\in R$, such that $\hat{p}\ge0$ on $K$, $K_i=\hat{p}^{-1}([0,i])$ is 
compact for each $i$ and $L:R\rightarrow\reals$ is $\T{K}$-continuous, $\integers[\frac{1}{2}]$-linear map, and $L(a^{2d})\ge0$ for all
$a\in R$, then there exists a Borel Measure $\mu$ on $K$ such that $\forall a\in R ~ L(a)=\int_K \hat{a}~d\mu$.
\end{crl}
\begin{proof}
Following the argument in the proof of Corollary \ref{M4Cmpt}, the map $\bar{L}:\hat{R}\rightarrow\reals$ is well-defined and has a $\mathcal{F}_K$-continuous
extension to the $\reals$-subalgebra $A$ of $C(K)$, generated by $\hat{R}$. Applying Theorem \ref{Haviland} to $\bar{L}$, $\hat{p}$ and $A$, the result follows.
\end{proof}

\section{Results for $\reals$-Algebras}\label{RAlg}
In this section we assume that $R$ is an $\reals$-algebra. First 
we consider closure of $\ringsop{R}{2d}$ with respect to any sub-multiplicative norm $\norm{}{}$ on $R$.
We prove that the closure of $\ringsop{R}{2d}$ with respect to the norm is equal to nonnegative elements over the global spectrum $\K{\norm{}{}}$ of 
$(R,\norm{}{})$.
Recall that the global spectrum of a topological $\reals$-algebra, also known as the Gelfand spectrum, is the set of all continuous elements of $\V{R}$. 

Furthermore, in the case of $\reals$-algebras, we generalize the conclusion of Theorem \ref{SosRngNC} to an arbitrary 
$\ringsop{R}{2d}$-module $M$.
\subsection*{Normed $\reals$-Algebras}

Suppose that $(R,\norm{}{})$ is a normed $\reals$-algebra, i.e., the norm satisfies the sub-multiplicativity condition $\norm{}{xy}\leq\norm{}{x}\norm{}{y}$ 
for all $x,y\in R$.
\begin{lemma}\label{CtsHom}
If $\alpha\in \K{\norm{}{}}$ then $|\alpha(x)|\leq\norm{}{x}$, for all $x\in R$.
\end{lemma}
\begin{proof}
In contrary suppose that $\exists x\in R$ such that $|\alpha(x)|>\norm{}{x}$. Then for $n\ge1$,
\[
    \norm{}{x^n}\leq\norm{}{x}^n\leq|\alpha(x)|^n=|\alpha(x^n)|.
\]
Therefore $\frac{|\alpha(x)|^n}{\norm{}{x}^n}\leq\frac{|\alpha(x^n)|}{\norm{}{x^n}}$. Since $\frac{|\alpha(x)|}{\norm{}{x}}>1$,
$\frac{|\alpha(x^n)|}{\norm{}{x^n}}\rightarrow\infty$ as $n\rightarrow\infty$. This contradicts the fact that $\alpha$ is $\norm{}{}$-continuous. So
\[
    \forall x\in R \quad |\alpha(x)|\leq\norm{}{x}.
\]
\end{proof}
\begin{lemma}\label{NormSo2d}
Let $d\ge1$ be an integer, $a\in R$ and $r>\norm{}{a}$. Then $(r\pm a)^{\frac{1}{2d}}\in\tilde{R}$, where $\tilde{R}$ is the completion of $(R,\norm{}{})$.
\end{lemma}
\begin{proof}
Let $\sum_{i=0}^{\infty}\lambda_it^i$ be the power series expansion on $(r\pm t)^{\frac{1}{2d}}$ about $t=0$. The series has the radius of convergence $r$.
Therefore, it converges for every $t$ with $|t|<r$. Since $r>\norm{}{a}$, $\sum_{i=0}^{\infty}\lambda_i\norm{}{a}^i<\infty$.
Note that $\norm{}{a^i}\leq\norm{}{a}^i$ for each $i=0,1,2,\cdots$, assuming $a^0=1$, so
\[
\begin{array}{lcl}
    \norm{}{\sum_{i=0}^{\infty}\lambda_ia^i} & \leq & \sum_{i=0}^{\infty}\lambda_i\norm{}{a^i}\\
     & \leq & \sum_{i=0}^{\infty}\lambda_i\norm{}{a}^i\\
     & < & \infty.
\end{array}
\]
This implies that $(1\pm a)^{\frac{1}{2d}}=\sum_{i=0}^{\infty}\lambda_ia^i\in\tilde{R}$.
\end{proof}
Let $A_{\norm{}{}}:=\{\norm{}{x}\pm x:x\in R\}$ and $M_{2d}$ be the $\ringsop{R}{2d}$-module generated by $A_{\norm{}{}}$. Clearly, $M_{2d}$ is archimedean and
hence $\K{M_{2d}}$ is compact. Note that
\[
\begin{array}{lcl}
    \alpha\in\K{M_{2d}} & \Leftrightarrow & \alpha(a)\ge0\quad\forall a\in M_{2d}\\
     & \Leftrightarrow & \norm{}{x}\pm\alpha(x)\ge0\quad\forall x\in R\\
     & \Leftrightarrow & |\alpha(x)|\leq\norm{}{x}\quad\forall x\in R\\
     & \Leftrightarrow & \alpha\textrm{ is }\norm{}{}\textrm{-continuous}.
\end{array}
\]
Therefore $\K{M_{2d}}$ is nothing but global spectrum of $(R,\norm{}{})$.

Theorem \ref{ClsrNrmdAlg} is the analogue of \cite[Theorem 4.3]{GMW} for normed algebras. Note that the fact that the Gelfand spectrum $\K{\norm{}{}}$ is 
compact is well-known (under additional assumptions)(see \cite[Theorem 2.2.3]{EKan}). However our proof is algebraic and based on the following result of T. Jacobi.

\noindent\textbf{Theorem.} Suppose $M \subseteq R$ is an archimedean $\sum R^{2d}$-module of $R$ for some integer $d\ge 1$. Then, $\forall$ $a\in R$, $$\hat{a}>0 \text{ on } \mathcal{K}_M \ \Rightarrow \ a\in M.$$

\begin{proof} See \cite[Theorem 4]{J} for every $d\ge1$ or \cite[Theorem 1.4]{MPut} for $d=1$.
\end{proof}

\begin{thm}\label{ClsrNrmdAlg}
Let $(R,\norm{}{})$ be a normed $\reals$-algebra and $d\ge1$ an integer. Then $\K{\norm{}{}}$ is compact and 
$\cl{\norm{}{}}{\ringsop{R}{2d}}=\pos{\K{\norm{}{}}}$.
\end{thm}
\begin{proof}
Since each $\alpha\in\K{\norm{}{}}$ is continuous and $\pos{\K{\norm{}{}}}=\bigcap_{\alpha\in\K{\norm{}{}}}\alpha^{-1}([0,\infty))$, $\pos{\K{\norm{}{}}}$ is
$\norm{}{}$-closed. Clearly $\ringsop{R}{2d}\subseteq\pos{\K{\norm{}{}}}$, therefore $\cl{\norm{}{}}{\ringsop{R}{2d}}\subseteq\pos{\K{\norm{}{}}}$.

For the reverse inclusion we have to show that if $a\in\pos{\K{\norm{}{}}}$ and $\epsilon>0$ are given, then $\exists b\in\ringsop{R}{2d}$ with
$\norm{}{a-b}\leq\epsilon$. Note that $\hat{a}+\frac{\epsilon}{2}$ is strictly positive on $\K{\norm{}{}}$. Since $\K{\norm{}{}}=\K{M_{2d}}$ and $M_{2d}$ is
archimedean, $\K{M_{2d}}$ is compact. 
By Jacobi's Theorem, $a+\frac{\epsilon}{2}\in M_{2d}$. So $a+\frac{\epsilon}{2}=\sum_{i=0}^k\sigma_is_i$, where $\sigma_i\in\ringsop{R}{2d}$,
$i=0,\dots,k$, $s_0=1$ and $s_i\in A_{\norm{}{}}$, $i=1,\dots,k$. Choose $\delta>0$ satisfying $(\sum_{i=0}^k\norm{}{\sigma_i})\delta\leq\frac{\epsilon}{2}$.
By Lemma \ref{NormSo2d} and continuity of the function $x\mapsto x^{2d}$ on $\tilde{R}$, there exists $r_i\in R$ such that
$\norm{}{\frac{\delta}{2}+s_i-r_i^{2d}}\leq\frac{\delta}{2}$, i.e., $\norm{}{s_i-r_i^{2d}}\leq\delta$, $i=1,\dots,k$.
Take $b=\sigma_0+\sum_{i=1}^k\sigma_ir_i^{2d}\in\ringsop{R}{2d}$. Then
\[
\begin{array}{lcl}
    \norm{}{a-b} & = & \displaystyle \norm{}{\sum_{i=1}^k\sigma_is_i-\sum_{i=1}^k\sigma_ir_i^{2d}-\frac{\epsilon}{2}}\\
     & \leq & \displaystyle \sum_{i=1}^k\norm{}{\sigma_i}\cdot\norm{}{s_i-r_i^{2d}}+\frac{\epsilon}{2}\\
     & \leq & \epsilon.
\end{array}
\]
This completes the proof.
\end{proof}
\begin{crl}\label{M4NAlge}
Let $(R,\norm{}{})$ be a normed $\reals$-algebra, $d\ge1$ an integer and $L:R\longrightarrow\reals$ a $\norm{}{}$-continuous linear functional. 
If for each $a\in R$, $L(r^{2d})\ge0$, then there exists a Borel measure $\mu$ on $\K{\norm{}{}}$ such that
\[
	L(a)=\int_{\K{\norm{}{}}}\hat{a}~d\mu\quad\forall a\in R.
\]
\end{crl}
\begin{proof}
Since $\K{\norm{}{}}$ is compact by Theorem \ref{ClsrNrmdAlg}, the conclusion follows by applying Theorem \ref{Haviland} for $X=\K{\norm{}{}}$ and $p=1$.
\end{proof}
\subsection*{Closures of $\sum R^{2d}$-modules in $\T{\V{R}}$}

\begin{thm}\label{ClMdl}
Let $M\subseteq R$ be a $\ringsop{R}{2d}$-module of $\reals$-algebra $R$ and $d\ge1$ an integer. Then $\cl{\T{\V{R}}}{M}=\pos{\K{M}}$.
\end{thm}
\begin{proof}
Since for each $m\in M$, $\hat{m}\ge0$ on $\K{M}$, we have $M\subseteq\pos{\K{M}}$. Each homomorphism $\alpha\in\V{R}$ is continuous, so
\[
    \pos{\K{M}}=\bigcap_{\alpha\in\K{M}}\alpha^{-1}([0,\infty)),
\]
is closed in $\T{\V{R}}$. Therefore $\cl{\T{\V{R}}}{M}\subseteq\pos{\K{M}}$. For the reverse inclusion, we show that for each $a\in\pos{\K{M}}$ and any open set
$U$ containing $a$, $U\cap M\neq\emptyset$. Hence $a\in\cl{\T{\V{R}}}{M}$.

For $a\in\pos{\K{M}}$ and an open set $U$ containing $a$, there exist $\alpha_1,\dots,\alpha_n\in\V{R}$ and $\epsilon>0$ such that
$U_{\alpha_1,\dots,\alpha_n}^{\epsilon}(a)\subseteq U$, where
\[
    U_{\alpha_1,\dots,\alpha_n}^{\epsilon}(a)=\{b\in R~:~\rho_{\alpha_i}(a-b)<\epsilon, i=1,\dots,n\}
\]
is a typical basic open set in $\T{\V{R}}$.\\
\textit{Claim}. For each $i=1,\dots,n$ there exists $t_i\in M$, $\alpha_i(a)=\alpha_i(t_i)$.

For each $i=1,\dots,n$, either $\alpha_i(a)\ge0$, or $\alpha_i(a)<0$.
If $\alpha_i(a)\ge0$, take $t_i=\alpha_i(a)\cdot 1_R$, which belonges to $M$.
Suppose that $\alpha_i(a)<0$. There exists $s_i\in M$ such that $\alpha_i(s_i)<0$. Otherwise, $\alpha_i\in\K{M}$ and hence $\alpha_i(a)\ge0$ which is a
contradiction. So $\frac{\alpha_i(a)}{\alpha_i(s_i)}>0$ and hence $t_i=\frac{\alpha_i(a)}{\alpha_i(s_i)}s_i\in M$. This completes the proof of the Claim.

For each $1\leq i,l\leq n$ set
\[
    p_{il}:=\prod_{\alpha_i(t_l)\neq\alpha_j(t_l)}\frac{t_l-\alpha_j(t_l)\cdot1}{\alpha_i(t_l)-\alpha_j(t_l)},
\]
if $\alpha_i(t_l)\neq\alpha_j(t_l)$ for some $1\leq j\leq n$, and $p_{il}=1$, if there is no such $j$. Then take $p_i=\prod_{l=1}^np_{il}$. Note that for each
$1\leq k\leq n$,
\[
\alpha_k(p_i)=\left\lbrace
\begin{array}{lr}
    0 & \textrm{if } \exists l\in\{1,\dots,n\}\quad\alpha_i(t_l)\neq\alpha_k(t_l),\\
    1 & \textrm{Otherwise}.
\end{array}
\right.
\]
Let $\lambda_i$ be the number of elements $k\in\{1,\dots,n\}$ such that $\alpha_i(t_l)=\alpha_k(t_l)$, for all $1\leq l\leq n$.
Take $p=\sum_{j=1}^{n}\frac{1}{\lambda_j}p_j^{2d}t_j$ which belonges to $M$. We have $\alpha_i(p)=\alpha_i(a)$ for $i=1,\dots,n$ and hence
$p\in U_{\alpha_1,\dots,\alpha_n}^{\epsilon}(a)$. Therefore $M\cap U_{\alpha_1,\dots,\alpha_n}^{\epsilon}(a)\neq\emptyset$, so $a\in \cl{\T{\V{R}}}{M}$, which
proves the reverse inclusion and hence $\cl{\T{\V{R}}}{M}=\pos{\K{M}}$.
\end{proof}
%
\section{Application to $R:=\rx$}\label{App2RX}
We are mainly interested in the special case of real polynomials. In
this case, $\rx$ is a free finitely generated commutative
$\reals$-algebra and hence every $\alpha\in\V{\rx}$ is completely
determined by $\alpha(X_i)$, $i=1,\dots,n$. So, $\V{\rx}=\reals^n$
with the usual euclidean topology.
\begin{crl}\label{SosPoly}
Let $K$ be a closed Zariski dense subset of $\reals^n$,
\begin{enumerate}
    \item{The family of multiplicative seminorms $\mathcal{F}_K$ induces a lmc Hausdorff topology $\T{K}$ on $\rx$ such that
    $\cl{\T{K}}{\ringsop{\rx}{2d}}=\pos{K}$.}
    \item{If $K$ is compact subset then $\norm{K}{f}=\sup_{x\in K}|f(x)|$ induces a norm on $\rx$ such that $\cl{\norm{K}{}}{\ringsop{\rx}{2d}}=\pos{K}$.}
\end{enumerate}
\end{crl}
\begin{proof}
Apply theorems \ref{SosRngNC} and \ref{SosRng} to get the asserted
equalities. The fact that $\T{K}$ is Hausdorff and $\norm{K}{}$ is
actually a norm, follows from Theorem \ref{ZD-H-Ker} and Remark
\ref{pullback} respectively.
\end{proof}
\begin{rem}~\\
(i) According to \cite[Theorem 3.1]{GKS}, the equality
$\cl{\norm{K}{}}{\ringsop{\rx}{2d}}=\pos{K}$ is equivalent to the
following: If $L$ is a linear functional on $\rx$ satisfying
$L(h^{2d})\ge0$ for all $h\in\rx$ and $\forall f\in\rx~|L(f)|\leq
C\norm{K}{f}$ for some real $C>0$, then there exists a positive
Borel measure $\mu$ on $K$, representing $L$:
\[
    \forall f\in\rx\quad L(f)=\int_Kf~d\mu.
\]
(ii) Reinterpreting the equation
$\cl{\T{K}}{\ringsop{\rx}{2d}}=\pos{K}$, we have the following: If
$L$ be a linear functional on $\rx$, satisfying $L(h^{2d})\ge0$ for
all $h\in\rx$ and for every $x\in K$, $\exists C_x>0$ such that
$|L(f)|\leq C_x|f(x)|$, then there exists a positive Borel measure
$\mu$ on $K$ representing $L$:
\[
    \forall f\in\rx\quad L(f)=\int_Kf~d\mu.
\]
\end{rem}

We now discuss the case when $K$ is not Zariski dense. By
Theorem \ref{ZD-H-Ker}, $\Phi$ is not injective and hence the
topology $\T{K}$ (or when $K$ is compact the topology induced by
$\norm{K}{a}=\sup_{\alpha\in K}|\hat{a}(\alpha)|$) will not be
Hausdorff. Let $K\subseteq\reals^n$ be given, then for $\epsilon>0$
the set $K^{(\epsilon)}:=\cl{}{\bigcup_{x\in K}N_{\epsilon}(x)}$ has
nonempty interior, in fact $K\subseteq(K^{(\epsilon)})^{\circ}$ and
hence $K^{(\epsilon)}$ is Zariski dense in $\reals^n$. If
$0<\epsilon_1\leq\epsilon_2$, then $K^{\epsilon_1}\subseteq
K^{\epsilon2}$ and the identity map
\[
    id_{\epsilon_2,\epsilon_1}:(\rx,\T{K^{\epsilon_2}})\rightarrow(\rx,\T{K^{\epsilon_1}})
\]
is continuous. Therefore the family
$\{(\rx,\T{K^{(\epsilon)}})_{\epsilon>0},(id_{\epsilon\delta})_{\delta\leq\epsilon}\}$
is an inverse system of lc and Hausdorff vector spaces. The inverse
limit of this system exists and is a lc and Hausdorff space
\cite[Section 2.6]{HJ}. Let
$(V,\tau_K)=\varprojlim\limits_{\epsilon>0}(\rx,\T{K^{(\epsilon)}})$.
Then $V=\rx$ and $\tau_K$ is a lc and Hausdorff topology.
\begin{thm}\label{XLCH}
For any closed $K\subseteq\reals^n$ and integer $d\ge1$, $\cl{\tau_K}{\ringsop{\rx}{2d}}$ is the cone $\tilde{\Pos}(K)$, consisting of those polynomials
which are non-negative over some open set, containing $K$.
\end{thm}
\begin{proof}
Since $\tau_K=\varprojlim\limits_{\epsilon>0}\T{K^{(\epsilon)}}$, we have
\[
    \cl{\tau_K}{\ringsop{\rx}{2d}}=\varprojlim\cl{\norm{K^{(\epsilon)}}{}}{\ringsop{\rx}{2d}}=\bigcap_{\epsilon}\pos{K^{(\epsilon)}}
\]
(See \cite[\S 4.4]{NBTop}) and $\bigcap_{\epsilon}\pos{K^{(\epsilon)}}=\tilde{\Pos}(K)$ by definition.
\end{proof}
\begin{rem}
Assuming $K$ is compact implies the compactness of each $K^{(\epsilon)}$. Therefore $\norm{K^{(\epsilon)}}{}$ is defined and and is a norm. Moreover, for
$0<\epsilon_1\leq\epsilon_2$, the identity map
\[
    id_{\epsilon_2,\epsilon_1}:(\rx,\norm{K^{(\epsilon_2)}}{})\rightarrow(\rx,\norm{K^{(\epsilon_1)}}{})
\]
is continuous by
$\norm{K^{(\epsilon_1)}}{}\leq\norm{K^{(\epsilon_2)}}{}$. So
$\{(\rx,\norm{K^{(\epsilon)}}{}),(id_{\epsilon\delta})_{\delta<\epsilon}\}$
is an inverse limit of normed spaces. The inverse limit topology
$\tau_K=\varprojlim\limits_{\epsilon>0}\norm{K^{(\epsilon)}}{}$
exists and is a lc Hausdorff topology on $\rx$. The equality
$\cl{\tau_K}{\ringsop{\rx}{2d}}=\tilde{\Pos}(K)$ can be verified
similar to Theorem \ref{XLCH}.
\end{rem}
\subsection{Comparison with sub-multiplicative norm topologies}
Now, let $\norm{}{}$ be a sub-multiplicative norm on $\rx$, i.e.
\[
	\norm{}{f\cdot g}\leq\norm{}{f}\cdot\norm{}{g}\quad\forall f,g\in\rx.
\]
\begin{prop}\label{SMNorm}
The $\norm{}{}$-topology is finer than $\norm{\K{\norm{}{}}}{}$-topology and
\[
    \cl{\norm{}{}}{\ringsop{\rx}{2d}}=\cl{\norm{\K{\norm{}{}}}{}}{\ringsop{\rx}{2d}}=\pos{\K{\norm{}{}}},
\]
for every integer $d\ge1$.
\end{prop}
\begin{proof}
To prove that $\norm{}{}$-topology is finer than $\norm{\K{\norm{}{}}}{}$-topology, we show $\norm{\K{\norm{}{}}}{f}\leq\norm{}{f}$. Note that
$\K{\norm{}{}}$ is compact by \ref{ClsrNrmdAlg}, so $\norm{\K{\norm{}{}}}{}$ is defined and by Lemma \ref{CtsHom},
\[
	\norm{\K{\norm{}{}}}{f} = \sup_{x\in\K{\norm{}{}}}|f(x)| \leq \sup_{x\in\K{\norm{}{}}}\norm{}{f} = \norm{}{f}.
\]
Therefore, the identity map $id:(\rx,\norm{}{})\rightarrow(\rx,\norm{\K{\norm{}{}}}{})$ is continuous and hence
$\norm{}{}$-topology is finer than $\norm{\K{\norm{}{}}}{}$-topology.
Moreover, $\cl{\norm{\K{\norm{}{}}}{}}{\ringsop{\rx}{2d}}=\pos{\K{\norm{}{}}}$ by Corollary \ref{SosPoly}.
\end{proof}

\begin{dfn}~
\begin{itemize}
	\item[(i)]{A function $\phi:\naturals^n\rightarrow\reals^{+}$ is called
an \textit{absolute value} if \[\phi(\underline{0})=1 \mbox{ and }
\forall s,t\in\naturals^n ~ \phi(s+t)\leq\phi(s)\phi(t)\>.\]}
	\item[(ii)]{For a polynomial $f=\sum_{s\in\naturals^n}f_s\ux^s$, let
$\norm{\phi}{f}:=\sum_{s\in\naturals^n}|f_s|\phi(s)$.}
\end{itemize}
\end{dfn} 

If $\phi>0$ on $\naturals^n$, then $\norm{\phi}{}$ defines a norm on
$\rx$. Berg and Maserick \cite{BCRBK,BM-EB} show that the closure of
$\sos$ with respect to the $\norm{\phi}{}$-topology is
$\pos{\K{\phi}}$, where
$\K{\phi}:=\{x\in\reals^n:|x^s|\leq\phi(s),~\forall
s\in\naturals^n\}=\K{\norm{\phi}{}}$, the Gelfand spectrum of $(\rx,\norm{\phi}{})$. 

If $\phi>0$ then $\K{\phi}$ has non-empty interior, and hence is Zariski dense. 
By \cite{GMW}, $\K{\phi}$ is compact. Hence $\norm{\K{\phi}}{}$ is defined and is a norm by Remark \ref{pullback}(1).

The following corollary to Proposition \ref{SMNorm} generalizes the result of Berg and Maserick \cite[Theorem 4.2.5]{BCRBK} to the closure of $\ringsop{\rx}{2d}$.
\begin{crl}\label{BergComp}
The $\norm{\phi}{}$-topology is finer than $\norm{\K{\phi}}{}$-topology and
\[
    \cl{\norm{\phi}{}}{\sos}=\cl{\norm{\K{\phi}}{}}{\sos}=\pos{\K{\phi}}.
\]
\end{crl}

\subsection{Comparison with Lasserre's topology}

Recently, Lasserre \cite{JL:K-Moment} considered the following
$\norm{w}{}$ on $\rx$:
\[
    \norm{w}{\sum_{s\in\naturals^n}f_s\ux^s} = \sum_{s\in\naturals^n}|f_s|w(s),
\]
where $w(s)=(2\lceil|s|/2\rceil)!$ and
$|s|=|(s_1,\dots,s_n)|=s_1+\cdots+s_n$. He proved that for any basic
semi-algebraic set $K\subseteq\reals^n$, defined by a finite set of
polynomials $S$, the closure of the quadratic module $M_S$ and the
preordering $T_S$ with respect to $\norm{w}{}$ are equal to
$\pos{K}$.
\begin{prop}
Let $K_S\subseteq\reals^n$  be a basic closed semi-algebraic set and $d\ge1$ an integer.
\begin{enumerate}
    \item{If $K_S$ is compact, then the $\norm{w}{}$-topology is finer than $\norm{K_S}{}$-topology and
\[
    \cl{\norm{w}{}}{M_S}=\cl{\norm{w}{}}{T_S}=\cl{\norm{K_S}{}}{\ringsop{\rx}{2d}}=\pos{K_S}.
\]
    }
    \item{$\norm{w}{}$-topology is finer than $\T{K_S}$ and
\[
    \cl{\norm{w}{}}{M_S}=\cl{\norm{w}{}}{T_S}=\cl{\T{K_S}}{\ringsop{\rx}{2d}}=\pos{K_S}.
\]
    }
\end{enumerate}
\end{prop}
\begin{proof}
(1) To show that $\norm{w}{}$-topology is finer than $\norm{K_S}{}$-topology, it suffices to prove that the formal identity map
\[
    id:(\rx,\norm{w}{})\longrightarrow(\rx,\norm{K_S}{})
\]
is continuous. Let $p_i:\reals^n\rightarrow\reals$ be the projection on $i$-th coordinate and
\[
    M=\max_{1\leq i\leq n}\{|p_i(x)|:x\in K_S\}.
\]
So, for each $s\in\naturals^n$ we have $\norm{K_S}{\ux^s}\leq M^{|s|}$. Also $w(s)\ge|s|!$ for all $s\in\naturals^n$. By Stirling's formula
$|s|!\sim\sqrt{2\pi}e^{(|s|+\frac{1}{2})\ln |s|-|s|}$,
we see that
\[
    \begin{array}{lclcl}
        \frac{\norm{K_S}{\ux^s}}{\norm{w}{\ux^s}} & \leq & \frac{M^{|s|}}{|s|!}&&\\
         & \sim & \frac{1}{\sqrt{2\pi}} e^{|s|(\ln M-\ln|s|+1)-\frac{1}{2}\ln|s|} & \xrightarrow{|s|\rightarrow\infty} & 0.
    \end{array}
\]
Therefore for some $N\in\naturals$, if $|s|>N$ then $\frac{\norm{K_S}{\ux^s}}{\norm{w}{\ux^s}}<1$, which shows that $id$ is bounded and hence continuous.
The asserted equality follows from Corollary \ref{SosRng} and \cite[Theorem 3.3]{JL:K-Moment}.

(2) It suffices to show that for any $x\in K_S$, the evaluation map $e_x(f)=f(x)$ is $\norm{w}{}$-continuous.
Since $\frac{|x^s|}{|s|!}\xrightarrow{|s|\rightarrow\infty}0$, we deduce that $\sup_{f\in\rx}\frac{|f(x)|}{\norm{w}{f}}$ is bounded. So, $e_x$ and hence $\rho_x$
is $\norm{w}{}$-continuous. Therefore any basic open set in $\T{K_S}$ is $\norm{w}{}$-open. The asserted equality follows from Theorem \ref{SosRngNC} and
\cite[Theorem 3.3]{JL:K-Moment}.
\end{proof}
%

\end{document}